\documentclass[10pt]{amsart}
\usepackage[utf8]{inputenc}
\usepackage{amsthm}
\usepackage{amsmath}

\newtheorem{theorem}{Theorem}

\newtheorem{lemma}[theorem]{Lemma}
\newtheorem{proposition}[theorem]{Proposition}
\newtheorem{corollary}[theorem]{Corollary}

\newcommand{\<}{\langle}
\renewcommand{\>}{\rangle}
\newcommand{\E}{\mathbb{E}}
\newcommand{\R}{\mathbb{R}}
\newcommand{\V}{\mathcal{V}}
\newcommand{\e}{\text{e}}
\newcommand{\x}{\otimes}
\allowdisplaybreaks
\title{The Heston stochastic volatility model in Hilbert space}
\author{Fred Espen Benth and Iben Cathrine Simonsen}
\address{Fred Espen Benth and Iben Cathrine Simonsen \\
Department of Mathematics \\
University of Oslo\\
P.O. Box 1053, Blindern\\
N--0316 Oslo, Norway}
\email[]{[fredb,ibens]\@@math.uio.no}
\date{\today}
\thanks{F. E. Benth acknowledges financial support from the project FINEWSTOCH, funded by the Norwegian Research Council.}

\begin{document}
\begin{abstract}
 We extend the Heston stochastic volatility model to a Hilbert space framework. 
 The tensor Heston stochastic variance process is defined as a tensor product of a Hilbert-valued Ornstein-Uhlenbeck process with itself.
 The volatility process is then defined by a Cholesky decomposition of the variance process.
 We define a Hilbert-valued Ornstein-Uhlenbeck process with Wiener noise perturbed by this stochastic volatility,
 and compute the characteristic functional and covariance operator of this process.
 This process is then applied to the modelling of forward curves in energy markets.
 Finally, we compute the dynamics of the tensor Heston volatility model when the generator is bounded, and study its projection down to the real line for 
comparison with the classical Heston dynamics.
\end{abstract}
\maketitle
\section{Introduction}

Ornstein-Uhlenbeck processes in Hilbert space has received some attention in the literature in recent years (see Applebaum~\cite{Apple}), one reason being that it is
a basic process for the dynamics of commodity forward prices (see Benth and Kr\"uhner~\cite{BK-HJM}). In the modelling of financial prices, the stochastic volatility 
dynamics plays an important role, and in this paper we propose an infinite dimensional version of the classical Heston model (see Heston~\cite{H}). 

On a separable Hilbert space $H$, an Ornstein-Uhlenbeck process $X(t)$ takes the form
$$
dX(t)=\mathcal C X(t)\,dt+\sigma\,dB(t),
$$
where $\mathcal C$ is some densely defined linear operator and $B$ is an $H$-valued Wiener process. Usually, $\sigma$ is some non-random bounded linear
operator on $H$, being a scaling of the noise which is referred to as the volatility. We propose to model $\sigma$ as a time-dependent stochastic process
with values in the space of bounded linear operators. More specifically, we consider a stochastic variance process $\mathcal V(t)$ being defined as the tensor product of
another Ornstein-Uhlenbeck process with itself, which will become a positive definite stochastic process in the space of Hilbert-Schmidt operators on $H$. We 
use its square root process as a volatility process $\sigma$ in the dynamics of $X$. Our construction is an extension of the classical Heston stochastic volatility model.   

If $H$ is some suitable space of real-valued functions 
on $\mathbb R_+$, the non-negative real numbers, and $\mathcal C=\partial/\partial x$, one can view $X(t,x)$ as the risk-neutral forward price at time $t\geq 0$ for some contract delivering a given commodity at time $t+x$. Such forward price models (with generalisations) have been extensively analysed in 
Benth and Kr\"uhner~\cite{BK-HJM}, being
stochastic models in the so-called Heath-Jarrow-Morton framework (see Heath, Jarrow and Morton~\cite{HJM}) with the Musiela parametrisation. The analysis
relates closely to a long stream of literature on forward rate modelling in fixed-income markets (see Filipovic~\cite{filipovic} and references therein). However, 
stochastic volatility models from the infinite dimensional perspective have not, to the best of our knowledge, been studied to any significant extent. An exception is the
paper by Benth, R\"udiger and S\"uss \cite{BRS}, who propose and analyse an infinite dimensional generalisation of the Barndorff-Nielsen and Shephard stochastic
volatility model (see Barndorff-Nielsen and Shephard~\cite{BNS}).   


As indicated, we define $\mathcal V(t)=Y(t)^{\otimes 2}$, where $Y$ is an $H$-valued Gaussian Ornstein-Uhlenbeck process. We prove several properties of 
the tensor Heston variance process $\mathcal V$, and show that the square-root process $\mathcal V^{1/2}$ is explicitly available. Moreover, we present
a family of Cholesky-type decompositions of $\mathcal V$, which will be our choice as stochastic volatility in the dynamics of $X$. We study 
probabilistic properties of both $\mathcal V$ and $X$, and specialize to the situation of a commodity forward market where we provide expressions for the implied 
covariance structure between forward prices with different times to maturity. In the situation when the Ornstein-Uhlenbeck process $Y$ is governed by a bounded 
generator, we can present a stochastic dynamics of $\mathcal V$ which can be related to the Heston model in the finite dimensional case. In particular, our model
is an alternative to the Wishart process of Bru~\cite{Bru}.

\subsection{Notation} We let $(\Omega, \mathcal{F}, \{\mathcal{F}_t\}_{t\geq 0}, \mathbb{P})$ be a filtered probability space and $H$ be a separable Hilbert space with inner product $\langle\cdot,\cdot\rangle$ and associated norm $|\,\cdot\,|$.
Furthermore, we let $L(H)$ denote the space of bounded linear operators from $H$ into itself, which is a Banach space with the operator
norm denoted $\|\cdot\|_{\text{op}}$. The adjoint of an operator $A\in L(H)$ is denoted
$A^*$. Furthermore, $\mathcal{H}=L_{HS}(H)$ denotes the space of Hilbert-Schmidt operators in $L(H)$. 
$\mathcal{H}$ is also a separable Hilbert space, and we denote its inner product by $\langle\langle\cdot,\cdot\rangle\rangle$ and the associated norm by $\|\,\cdot\,\|$. 

\section{The tensor Heston stochastic variance process}
Let $\{W(t)\}_{t\geq 0}$ be an $\mathcal F_t$-Wiener process in $H$ with covariance operator $Q_W\in L(H)$, where $Q_W$ is a symmetric and positive definite trace class operator.
Define the Ornstein-Uhlenbeck process $\{Y(t)\}_{t\geq 0}$ in $H$ by
\begin{equation} \label{Y}
 dY(t) = \mathcal{A} Y(t)\, dt + \eta\, dW(t), \mspace{40mu} Y(0)=Y_0 \in H,
\end{equation}
where $\mathcal{A}$ is a densely defined operator on $H$ generating a $C_0$-semigroup $\{\mathcal{U}(t)\}_{t\geq 0}$, and $\eta\in L(H)$.
From Peszat and Zabczyk~\cite[Sect.~9.4]{PZ}, the unique mild solution of \eqref{Y} is given by
\begin{equation}\label{eq:tensor-heston-Y-process-mild}
 Y(t) = \mathcal{U}(t) Y_0 + \int_0^t \mathcal{U}(t-s)\eta\, dW(s),
\end{equation}
for $t\geq 0$.
The next lemma gives the characteristic functional of $Y(t)$.
\begin{lemma}
\label{lemma:gaussian-Y}
 For $f\in H$ we have
 \[
  \mathbb{E}\Big[\exp \left(\mathrm{i}\langle Y(t),f\rangle\right)\Big] = \exp \left(\mathrm{i}\langle\mathcal{U}(t)Y_0,f\rangle
-\frac{1}{2}\langle\int_0^t\mathcal U(s)\eta Q_W\eta^* \mathcal{U}^*(s)\,ds f,f\rangle\right),
 \]
where the integral on the right-hand side is the Bochner integral on $L(H)$. 
\end{lemma}
\begin{proof}
From the mild solution of  $\{Y(t)\}_{t\geq 0}$ in \eqref{eq:tensor-heston-Y-process-mild}, we find
\begin{align*}
\langle Y(t),f\rangle&=\langle\mathcal U(t) Y_0,f\rangle+\langle\int_0^t\mathcal U(t-s)\eta\,dW(s),f\rangle \\
&=\langle\mathcal U(t)Y_0,f\rangle+\int_0^t\langle\eta^*\mathcal U^*(t-s) f,dW(s)\rangle.
\end{align*}
Hence, from the Gaussianity and independent increment property of the Wiener process, 
\begin{align*}
\E\left[\exp\left(\mathrm{i}\langle Y(t),f\rangle\right)\right]=\exp\left(\mathrm{i}\langle\mathcal U(t) Y_0,f\rangle-\frac12\int_0^t\langle
\mathcal U(t-s)\eta Q_W\eta^*\mathcal U^*(t-s)f,f\rangle\,ds\right).
\end{align*}
As $\{\mathcal U(t)\}_{t\geq 0}$ is a $C_0$-semigroup, its operator norm satisfies an exponential growth bound in time by the 
Hille-Yoshida Theorem (see Engel and Nagel~\cite[Prop.~I.5.5]{EN}). Hence, 
the Bochner integral $\int_0^t\mathcal U(s)\eta Q_W\eta^*\mathcal U^*(s)\,ds$ is well-defined, and the result follows.
\end{proof}
From the lemma above we conclude that $\{Y(t)\}_{t\geq 0}$ is an $H$-valued Gaussian process with mean $\mathcal{U}(t)Y_0$ and 
covariance operator
$$
Q_{Y(t)}=\int_0^t\mathcal{U}(s)\eta\mathcal{Q}_W\eta^*\mathcal{U}(s)^*\,ds.
$$
Following Applebaum~\cite{Apple}, $\{Y(t)\}_{t\geq 0}$ admits an invariant Gaussian distribution with zero mean if the $C_0$-semigroup 
$\{\mathcal U(t)\}_{t\geq 0}$ is exponentially stable. The covariance operator for the invariant mean zero Gaussian distribution of $\{Y(t)\}_{t\geq 0}$ then becomes
$$
Q_Y=\int_0^{\infty}\mathcal{U}(s)\eta\mathcal{Q}_W\eta^*\mathcal{U}(s)^*\,ds.
$$  

We define the \emph{tensor Heston stochastic variance process} $\{\V(t)\}_{t\geq 0}$ by
\begin{equation}
 \V(t):=Y(t)^{\x2},
\end{equation}
where we recall the tensor product to be $f\x g:=\<f,\cdot\,\>g$ for $f,g\in H$. By the Cauchy-Schwartz inequality, it follows straightforwardly that
$f\x g\in L(H)$. 
Hence, the tensor Heston 
stochastic variance process $\{\mathcal V(t)\}_{t\geq 0}$ defines an $\mathcal F_t$-adapted stochastic process in $L(H)$. The next proposition
shows that $\{\mathcal V(t)\}_{t\geq 0}$ defines a family of symmetric, positive definite Hilbert-Schmidt operators.
\begin{proposition}
\label{prop:Vsymposdef}
 It holds that $\V(t)\in \mathcal{H}$ for all $t\geq 0$.  Furthermore, $\V(t)$ is a symmetric and positive definite operator. 
\end{proposition}
\begin{proof}
Let $\{e_n\}_{n=1}^{\infty}$ be an orthonormal basis (ONB) of $H$. By Parseval's identity applied twice,
\begin{align*}
\|\mathcal{V}(t)\|^2&=\sum_{n=1}^{\infty}|\mathcal{V}(t)e_n|^2 =\sum_{n=1}^{\infty}|Y^{\otimes 2}(t)e_n|^2
=\sum_{n=1}^{\infty}|Y(t)|^2\<Y(t),e_n\>^2 =|Y(t)|^4.
\end{align*}
Since $Y(t)\in H$ for every $t\geq 0$, the first conclusion of the proposition follows. 

We find for $f,g\in H$ that 
$$
\<\mathcal{V}(t)f,g\>=\<\<Y(t),f\> Y(t),g\>=\<Y(t),f\> \<Y(t),g\> = \<f,\<Y(t),g\> Y(t)\>=\<f,\mathcal{V}(t)g\>.
$$
Moreover, with $f=g$,
$$
\<\mathcal{V}(t)f,f\>=\<Y(t),f\>^2\geq 0.
$$
This proves the second part.
\end{proof}
The proposition shows that $\|\mathcal V(t)\|=|Y(t)|^2$ for all $t\geq 0$. The Gaussianity of the process $\{Y(t)\}_{t\geq 0}$ implies that 
the real-valued stochastic process $\{\|\mathcal V(t)\|\}_{t\geq 0}$ has finite exponential moments up to a certain order:
\begin{lemma}
It holds that 
$$
\mathbb E[\exp(\theta\|\mathcal V(t)\|)]\leq\frac{\e^{2\theta|\mathcal U(t)Y_0|^2}}{\sqrt{1-4\theta k}}
$$
for $0\leq\theta\leq 1/4k$ and $k=\mathbb E[|\int_0^t\mathcal U(t-s)\eta\,dW(s)|^2]<\infty$.
\end{lemma}
\begin{proof}
From Prop.~\ref{prop:Vsymposdef}, $\|\mathcal V(t)\|=|Y(t)|^2$, and then  by the triangle inequality
$$
\|\mathcal V(t)\|\leq 2|Y(t)-\mathcal U(t)Y_0|^2+2|\mathcal U(t)Y_0|^2.
$$
From the mild solution of $Y(t)$ in \eqref{eq:tensor-heston-Y-process-mild}, 
$$
Y(t)-\mathcal U(t)Y_0=\int_0^t\mathcal U(t-s)\eta\,dW(s),
$$
which is a centered Gaussian random variable. Hence, Fernique's Theorem (see Fernique~\cite{Fernique} or Thm. 3.31 in Peszat and Zabczyk~\cite{PZ}) 
implies that $k=\mathbb E[|\int_0^t\mathcal U(t-s)\eta\,dW(s)|^2]<\infty$ and
$$
\mathbb E\left[\exp\left(\theta\|\mathcal V(t)\|\right)\right]\leq \e^{2\theta|\mathcal U(t)Y_0|^2}\mathbb E\left[\exp\left(2\theta|\int_0^t\mathcal U(t-s)\eta\,dW(s)|^2\right)\right]\leq \e^{2\theta|\mathcal U(t)Y_0|^2}\frac1{\sqrt{1-4\theta k}}
$$
for $0\leq \theta\leq 1/4k$. 
\end{proof}
From this lemma we can conclude that all moments of the real-valued random variable $\|\mathcal V(t)\|$ are finite, as $\|\mathcal V(t)\|\leq\exp(s\|\mathcal V(t)\|)$ for arbitrary small $s>0$. 

If $f,g\in H$, then we see that 
$$
\langle\langle\mathcal V(t),f\otimes g\rangle\rangle=\langle Y(t),f\rangle\langle Y(t),g\rangle.
$$
Recalling Lemma~\ref{lemma:gaussian-Y}, $Y(t;f):=\langle Y(t),f\rangle$ is normally distributed with mean 
$\mathbb E[Y(t;f)]=\langle\mathcal U(t)Y_0,f\rangle$ and variance
$v(f):=\textnormal{Var}(Y(t;f))=\int_0^t |Q^{1/2}_W\eta^*\mathcal U^*(s)f|^2\,ds$. Moreover,
\begin{align*}
c(f,g):&=\textnormal{Cov}(Y(t;f),Y(t;g)) \\
&=\mathbb E\left[\langle\int_0^t\mathcal U(t-s)\eta\,dW(s),f\rangle\langle\int_0^t\mathcal U(t-s)\eta\,dW(s),g\rangle\right] \\
&=\int_0^t\langle Q^{1/2}_W\eta^*\mathcal U^*(s)f,Q^{1/2}_W\eta^*\mathcal U^*(s) g\rangle\,ds.
\end{align*}
A straightforward (but tedious) calculation reveals that the characteristic functional of $\mathcal V(t)$ evaluated at $f\otimes g$ becomes,
$$
\mathbb E\left[\e^{\mathrm{i}\langle\langle\mathcal V(t),f\otimes g\rangle\rangle}\right]=\left(1+v(f)v(g)-c^2(f,g)-2\mathrm{i}c(f,g)\right)^{-1/2},
$$
where we have assumed $Y_0=0$ for simplicity. This characteristic functional is related to a noncentral $\chi^2$-distribution with one degree of freedom. 
We recall that the variance process in the classical Heston model has a noncentral $\chi^2$-distribution (see Heston~\cite{H}).   

Since $\V(t)$ is symmetric and positive definite, we can define its square root $\V^{1/2}(t)$, which turns out to have an explicit expression.
\begin{proposition}
 The square root process of $\{\mathcal V(t)\}_{t\geq 0}$ is given by
$$
\V^{1/2}(t)=\left\{\begin{array}{cl}|Y(t)|^{-1}\V(t), & Y(t)\neq 0 \\ 0, & Y(t)=0.\end{array}\right. 
$$
\end{proposition}
\begin{proof}
If $Y(t)=0$, it follows that $\mathcal V(t)=0$, and thus also $\mathcal V^{1/2}(t)=0$. Assume $Y(t)\neq 0$. 
Let $f\in H$. Define $\mathcal{T}(t)=|Y(t)|^{-1}\mathcal{V}(t)$, which is symmetric and positive definite by 
Prop.~\ref{prop:Vsymposdef}. Then,
$$
\mathcal{T}^2(t)f=\mathcal{T}(t)(|Y(t)|^{-1}\mathcal{V}(t)f)=|Y(t)|^{-1}\mathcal{T}(t)(\mathcal{V}(t)f)=|Y(t)|^{-2}\mathcal{V}^2(t)f\,.
$$
But,
\begin{align*}
\mathcal{V}^2(t)f &= \mathcal{V}(t)(Y^{\otimes 2}(t)f) =\<Y(t),f\>_H\mathcal{V}(t)Y(t) =\<Y(t),f\>_H \<Y(t),Y(t)\>_H Y(t) =|Y(t)|^2\mathcal{V}(t)f\,.
\end{align*}
Hence, $\mathcal{T}^2(t)=\mathcal{V}(t)$, and the result follows.
\end{proof}

Consider for a moment the operator $\mathcal F:H\rightarrow\mathcal H$ defined as $f\mapsto\mathcal F(f):=|f|^{-1}f^{\x 2}$ for $f\neq 0$ and
$\mathcal F(0)=0$. 
\begin{lemma}
\label{lem:prel-F}
The operator $\mathcal F:H\rightarrow\mathcal H$ is locally Lipschitz continuous.
\end{lemma}
\begin{proof}
It holds for $f\neq 0$ that 
$$
\|\mathcal F(f)\|=|f|^{-1}\|f^{\x 2}\|=|f|^{-1}|f|^2=|f|.
$$
Hence, if $f\rightarrow0$ in $H$, then $\mathcal F(f)\rightarrow 0$ in $\mathcal H$, so $\mathcal F$ is continuous in zero. Next, suppose that 
$f,g\in H$ are both non-zero. Then, by a simple application of the triangle inequality and its reverse, 
\begin{align*}
\|\mathcal F(f)-\mathcal F(g)\|&=\| |f|^{-1}f^{\x 2}-|g|^{-1}g^{\x 2}\| \\
&\leq \left| |f|^{-1}-|g|^{-1}\right| \|f^{\x 2}\|+|g|^{-1}\|f^{\x 2}-g^{\x 2}\| \\
&= |f|^{-1}|g|^{-1}\left| |f|-|g|\right| |f|^2+|g|^{-1}\|f^{\x 2}-g^{\x 2}\| \\
&\leq |f||g|^{-1}|f-g|+|g|^{-1}\|f^{\x 2}-g^{\x 2}\|.
\end{align*}
Again, by triangle inequality and the elementary inequality $(x+y)^2\leq 2x^2+2y^2$, we find for an ONB $\{e_n\}_{n\in\mathbb N}$ in $H$,
\begin{align*}
\|f^{\x 2}-g^{\x 2}\|^2&=\sum_{n=1}|(f^{\x 2}-g^{\x 2})e_n|^2 \\
&=\sum_{n=1}^{\infty}|\langle f,e_n\rangle f-\langle g,e_n\rangle g|^2 \\
&\leq 2\sum_{n=1}^{\infty}\langle f,e_n\rangle^2|f-g|^2+2\sum_{n=1}^{\infty}|g|^2\langle f-g,e_n\rangle^2 \\
&=2|f-g|^2|f|^2+2|g|^2|f-g|^2.
\end{align*}
Therefore,
$$
\|\mathcal F(f)-\mathcal F(g)\|\leq\left(|f||g|^{-1}+\sqrt{2}(1+|f| |g|^{-1})^{1/2}\right)|f-g|
$$
which shows locally Lipschitz continuity of $\mathcal F$. 
\end{proof}
By a result of 
Kotelenez~\cite{Kotelenez} (see Peszat and Zabczyk~\cite[Thm.~9.20]{PZ}), there exists a continuous version of $\{Y(t)\}_{t\geq 0}$ in
\eqref{eq:tensor-heston-Y-process-mild} if the $C_0$-semigroup $\{\mathcal U(t)\}_{t\geq 0}$ is quasi-contractive, that is, if for some constant $\beta$, $\|\mathcal U(t)\|_{\text{op}}\leq \e^{\beta t}$ for all $t\geq 0$. 
Thus, from Lemma~\ref{lem:prel-F}, we can conclude that there exists a version of $\{\mathcal V^{1/2}(t)\}_{t\geq 0}$ (namely defined by the version of $\{Y(t)\}_{t\geq 0}$ with continuous paths) which has continuous paths in $\mathcal H$, when $\{\mathcal U(t)\}_{t\geq 0}$ is quasi-contractive. 
We remark that $|Y(t)|>0$ $a.s.$ This holds true since by Parseval's identity 
$$
|Y(t)|^2=\sum_{n=1}^{\infty}\<Y(t),e_n\>^2\,,
$$ 
for $\{e_n\}_{n=1}^{\infty}$ an ONB of $H$. By Lemma~\ref{lemma:gaussian-Y}, $\<Y(t),e_n\>$ is a Gaussian random variable for all $n$, and 
$P\left(\<Y(t),e_n\>=0\right)=0$. If $|Y(t)|=0$, then we must have $\<Y(t),e_n\>^2=0$ for all $n$. But this happens with probability zero, and 
it follows that $P(|Y(t)|=0)=0$. 

We move our attention to a Cholesky-type of decomposition of the tensor Heston stochastic variance process $\{\mathcal V(t)\}_{t\geq 0}$. To this end, 
introduce an $\mathcal F_t$-adapted $H$-valued stochastic process $\{Z(t)\}_{t\geq 0}$ for which $|Z(t)|=1$, i.e., a process living on the unit ball of $H$. 
Define the operator $\Gamma_Z(t) :H\rightarrow H$ for $t\geq 0$ by
\begin{equation}
\label{def:gammaZ}
 \Gamma_Z(t) := Z(t)\x Y(t). 
\end{equation}
The following lemma collects the elementary properties of this operator-valued process.
\begin{lemma}
\label{lemma:gamma-prop}
 It holds that $\{\Gamma_Z(t)\}_{t\geq 0}$ is an $\mathcal F_t$-adapted stochastic process with values in $\mathcal H$. 
\end{lemma}
\begin{proof}
By definition $\Gamma_Z(t)$ becomes a linear operator, where boundedness follows readily from the Cauchy-Schwartz inequality. 
For an ONB $\{e_n\}_{n\in\mathbb N}$ in $H$, we have from Parseval's identity
 \begin{equation*}
 \| \Gamma_Z(t)\|^2 = \sum_{n=1}^{\infty}|\Gamma_Z(t) e_n|^2 = \sum_{n=1}^{\infty} |\langle Z(t),e_n\rangle Y(t)|^2 = |Y(t)|^2 |Z(t)|^2=|Y(t)|^2< \infty.
 \end{equation*}
The $\mathcal F_t$-measurability follows by assumption on $Z(t)$ and definition of $Y(t)$. The proof is complete.
\end{proof}
We notice that with the convention $0/0=1$, we can define $Z(t)=Y(t)/|Y(t)|$ and recover $\mathcal V^{1/2}(t)=\Gamma_{Z}(t)$ for all
$t\geq 0$ such that $Y(t)\neq 0$. We show that 
for general unitary processes, $\{\Gamma_Z(t)\}_{t\geq 0}$ defines a Cholesky decomposition of the tensor Heston stochastic variance process:
\begin{proposition}
\label{prop:tensor-heston-cholesky}
The tensor Heston stochastic variance process $\{\V(t)\}_{t\geq 0}$ can be decomposed as
 \begin{align*}
  \V(t) &= \Gamma_Z(t) \Gamma^*_Z(t),
 \end{align*}
for all $t\geq 0$, where $\Gamma_Z^*(t)=Y(t)\x Z(t)$. 
\end{proposition}
\begin{proof}
 Since, for any $f,g\in H$ and $t\geq 0$, 
 \begin{align*}
  \langle\Gamma^*_Z(t)f,g\rangle&= \langle f,\Gamma_Z(t) g\rangle=\langle f,\langle Z(t), g\rangle Y(t)\rangle=\langle Z(t), g\rangle\langle f, Y(t)\rangle = \langle\langle f, Y(t)\rangle Z(t), g\rangle,
 \end{align*}
we have that $\Gamma_Z^*(t)=Y(t)\x Z(t)$.
It follows that for any $f\in H$,
\begin{align*}
 \Gamma_Z(t)\Gamma_Z^*(t)f &= \Gamma_Z(t) ( \langle Y(t), f\rangle Z(t))=
\langle Y(t),f\rangle|Z(t)|^2Y(t)
=Y^{\x 2}(t)(f)=\mathcal V(t)(f).
\end{align*}
The result follows.
\end{proof}
A simple choice of an $H$-valued stochastic process $\{Z(t)\}_{t\geq 0}$ is $Z(t)=\gamma$, where $\gamma\in H$ with $|\gamma|=1$.

\section{Ornstein-Uhlenbeck process with stochastic volatility}

Define the $H$-valued Ornstein-Uhlenbeck process $\{X(t)\}_{t\geq 0}$ by 
\begin{equation} \label{X}
 dX(t) = \mathcal C X(t)dt + \Gamma_Z(t) dB(t), \mspace{40mu} X(0)=X_0 \in H,
\end{equation}
where $\mathcal C$ is a densely defined operator on $H$ generating a $C_0$-semigroup $\mathcal{S}$,
and $\{B(t)\}_{t\geq 0}$ is a Wiener process in $H$ with covariance operator $Q_B\in L(H)$ (i.e., $Q_B$ is a symmetric and positive definite trace class operator).
We assume that $\{B(t)\}_{t\geq 0}$ is independent of $\{W(t)\}_{t\geq 0}$ and recall $\{\Gamma_Z(t)\}_{t\geq 0}$ from \eqref{def:gammaZ}.

The next lemma validates the existence of the stochastic integral in \eqref{X}:
\begin{lemma}
\label{lem:finite-integral-norm}
It holds that
\begin{align*}
 \mathbb{E}\left[\int_0^t \lVert \Gamma_Z(s)Q_B^{1/2} \rVert^2 ds\right] &\leq \mathrm{Tr}(Q_B)\,\mathbb{E}\left[\int_0^t | Y(s)|^2 ds \right]<\infty.
\end{align*}
\end{lemma}
\begin{proof}
Let $\{e_n\}_{n\in\mathbb N}$ be an ONB in $H$. By Parseval's identity, we have
\begin{align*}
 \lVert \Gamma_Z(s) Q_B^{1/2}\rVert^2 &= \sum_{n=1}^{\infty}|\Gamma_Z(s) Q_B^{1/2} e_n|^2 \\
 &= \sum_{n=1}^{\infty}|\<Z(t),Q_B^{1/2}e_n\>_H Y(s)|^2  \\
 &= |Y(s)|^2 \sum_{n=1}^{\infty} \<Z(t),Q_B^{1/2}e_n\>^2 \\
 &= |Y(s)|^2 \sum_{n=1}^{\infty} \<Q_B^{1/2}Z(t),e_n\>^2 \\
 &= |Y(s)|^2 |Q_B^{1/2}Z(t)|^2.
\end{align*}
As $Q_B$ is a symmetric, positive definite trace class operator, we can find an ONB of eigenvectors $\{v_n\}_{n\in\mathbb N}$ in $H$ with corresponding positive eigenvalues $\{\lambda_n\}_{n\in\mathbb N}$ of $Q_B$, such that $Q_Bv_n=\lambda_nv_n$ for all $n\in\mathbb N$ and therefore
$\text{Tr}(Q_B)=\sum_{n=1}^{\infty}\lambda_n<\infty$. We have by Parseval's identity and Cauchy-Schwartz' inequality, 
\begin{align*}
|Q_B^{1/2}Z(t)|^2&=\langle Q_B^{1/2}Z(t),Q_B^{1/2}Z(t)\rangle \\
&=\langle Q_BZ(t),Z(t)\rangle \\
&=\sum_{n=1}^{\infty}\langle Z(t),Q_B v_n\rangle\langle Z(t),v_n\rangle \\
&=\sum_{n=1}^{\infty}\lambda_n\langle Z(t),v_n\rangle^2 \\
&\leq\sum_{n=1}^{\infty}\lambda_n|Z(t)|^2|v_n|^2 \\
&=\text{Tr}(Q_B),
\end{align*}
since by assumption $|Z(t)|=|v_n|=1$. Next we show that $\E[\int_0^t|Y(s)|^2\,ds]<\infty$. From the expression 
in \eqref{eq:tensor-heston-Y-process-mild}, it follows from an elementary inequality that
\begin{align*}
\E\left[|Y(t)|^2\right]&=\E\left[|\mathcal U(t)Y_0+\int_0^t\mathcal U(t-s)\eta\,dW(s)|^2\right] \\
&\leq 2|\mathcal U(t)Y_0|^2+2\E\left[|\int_0^t\mathcal U(t-s)\eta\,dW(s)|^2\right] \\
&=2|\mathcal U(t)Y_0|^2+2\int_0^t\|\mathcal U(t-s)\eta Q_W^{1/2}\|^2\,ds,
\end{align*}
where the last equality is a consequence of the It\^o isometry. The Hille-Yoshida Theorem (see Engel and Nagel~\cite[Prop.~I.5.5]{EN}) implies that
$\|\mathcal U(t)\|_{\text{op}}\leq K\exp(w t)$ for constants $K>0$ and $w$. Thus,
\begin{align*}
\E\left[|Y(t)|^2\right]&\leq 2K^2\e^{2wt}|Y_0|^2+2\int_0^tK^2\e^{2w(t-s)}\,ds\|\eta\|_{\text{op}}^2\|Q^{1/2}_W\|^2.
\end{align*}
Finally, we observe that $\|Q^{1/2}_W\|^2=\text{Tr}(Q_W)<\infty$, and hence the lemma follows.
\end{proof}
The integral $\int_0^t \Gamma_Z(s)dB(s)$ is well-defined, and therefore according to Peszat and Zabczyk~\cite[Sect.~9.4]{PZ} \eqref{X} has a unique mild solution given by
\begin{equation}
\label{OU-mild-sol}
 X(t) = \mathcal{S}(t)X_0 + \int_0^t \mathcal{S}(t-s)\Gamma_Z(s)\,dB(s), \qquad t\geq 0.
\end{equation}
We remark in passing that the stochastic integral is well-defined in \eqref{OU-mild-sol} since $\mathcal S(t)\in L(H)$ with an operator norm
which is growing at most exponentially by the Hille-Yoshida Theorem (see Engel and Nagel~\cite[Prop.~I.5.5]{EN}).

We analyse the characteristic functional of $\{X(t)\}_{t\geq 0}$. To this end, denote by 
$\{\mathcal F_t^Y\}_{t\geq 0}$ the filtration generated by $\{Y(t)\}_{t\geq 0}$. 
\begin{proposition}
\label{prop:ch-funct-OU}
Assume that the process 
$\{Z(t)\}_{t\geq 0}$ in $\{\Gamma_Z(t)\}_{t\geq 0}$ defined in \eqref{def:gammaZ} is $\mathcal F_t^Y$-adapted.  It holds for any $f\in H$ 
\begin{align*}
  \mathbb{E}\left[\e^{\mathrm{i}\left<X(t),f\right>}\right] &=\e^{\mathrm{i}\left<\mathcal{S}(t)X_0,f\right>}\mathbb{E}\left[\exp\left(-\frac{1}{2}
\langle\int_0^t|Q_B^{1/2}Z(s)|^2\mathcal S(t-s)\mathcal V(s)\mathcal S^*(t-s)\,dsf,f\rangle\right)\right],
 \end{align*}
where the $ds$-integral on the right-hand side is a Bochner integral in $L(H)$.
\end{proposition}
\begin{proof}
With $f\in H$ we get from \eqref{OU-mild-sol},
$$
\E\left[\e^{\mathrm{i}\langle X(t),f\rangle}\right]=\e^{\mathrm{i}\langle\mathcal S(t)X_0,f\rangle}\E\left[\e^{\mathrm{i}\langle\int_0^t
\mathcal S(t-s)\Gamma_Z(s)\,dB(s),f\rangle}\right].
$$
Recall that $\{B(t)\}_{t\geq 0}$ and $\{W(t)\}_{t\geq 0}$ are independent. Since $\{Z(t)\}_{t\geq 0}$ is assumed $\mathcal F_t^Y$-adapted,
we will have that $\{Z(t)\}_{t\geq 0}$ and $\{Y(t)\}_{t\geq 0}$, and therefore $\{\Gamma_Z(t)\}_{t\geq 0}$, are independent of $\{B(t)\}_{t\geq 0}$.
By the tower property of conditional expectation and the Gaussianity of $\int_0^t\mathcal S(t-s)\Gamma_Z(s)\,dB(s)$ conditional on $\mathcal F_t^Y$,
\begin{align*}
\E\left[\e^{\mathrm{i}\langle\int_0^t
\mathcal S(t-s)\Gamma_Z(s)\,dB(s),f\rangle}\right]&=\E\left[\E\left[\e^{\mathrm{i}\langle\int_0^t\mathcal S(t-s)\Gamma_Z(s)\,dB(s),f\rangle}\vert\mathcal F_t^Y\right]\right] \\
&=\E\left[\e^{-\frac12\int_0^t|Q_B^{1/2}\Gamma_Z^*(s)\mathcal S^*(t-s)f|^2\,ds}\right].
\end{align*} 
Recalling from Prop.~\ref{prop:tensor-heston-cholesky}, we have $\Gamma_Z^*(s)=Y(s)\x Z(s)$. Hence,
$$
Q^{1/2}_B\Gamma_Z^*(s)(\mathcal S^*(t-s)f)=Q^{1/2}_B(\langle\mathcal S^*(t-s)f,Y(s)\rangle Z(s))=\langle Y(s), \mathcal S^*(t-s)f\rangle Q^{1/2}_BZ(s),
$$
and
\begin{align*}
|Q^{1/2}_B\Gamma_Z^*(s)\mathcal S^*(t-s)f|^2&=\langle Y(s),\mathcal S^*(t-s)f\rangle^2|Q^{1/2}_BZ(s)|^2 \\
&=\langle\mathcal V(s)\mathcal S^*(t-s)f,\mathcal S^*(t-s)f\rangle|Q^{1/2}_BZ(s)|^2 \\
&=\langle\mathcal S(t-s)\mathcal V(s)\mathcal S^*(t-s)f,f\rangle|Q^{1/2}_BZ(s)|^2.
\end{align*}
The proof is complete.
\end{proof}
From the proposition we see that $\{X(t)\}_{t\geq 0}$ is a Gaussian process conditional on $\mathcal F_t^Y$, with mean $\mathcal S(t)X_0$. The 
covariance operator $Q_{X(t)}$ of $X(t)$, defined by the relationship 
\begin{equation}
\langle Q_{X(t)}f,g\rangle=\E\left[\langle X(t)-\E[X(t)],f\rangle\langle X(t)-\E[X(t)],g\rangle\right],
\end{equation}
for $f,g\in H$, can be computed as follows: since $X(t)-\E[X(t)]=\int_0^t\mathcal S(t-s)\Gamma_Z(s)\,dB(s)$, and for a fixed $T>0$, the process
$t\mapsto\int_0^t\mathcal S(T-s)\Gamma_Z(s)\,dB(s)\, t\leq T$  is an $\mathcal F_t$-martingale, it follows from 
Peszat and Zabczyk~\cite[Thm. 8.7 (iv)]{PZ},
\begin{align*}
 &\E\left[\langle\int_0^t\mathcal S(T-s)\Gamma_Z(s)\,dB(s),f\rangle\langle\int_0^t\mathcal S(T-s)\Gamma_Z(s)\,dB(s),g\rangle\right] \\
&\qquad\qquad=
\E\left[\left\langle\left(\int_0^t\mathcal S(T-s)\Gamma_Z(s)\,dB(s)\right)^{\x 2}f,g\right\rangle\right] \\
&\qquad\qquad=\E\left[\langle\int_0^t\mathcal S(T-s)\Gamma_Z(s)Q_B\Gamma_Z^*(s)\mathcal S^*(T-s)\,ds f,g\rangle\right] \\
&\qquad\qquad=\langle\int_0^t\mathcal S(T-s)\E\left[\Gamma_Z(s)Q_B\Gamma_Z^*(s)\right]\mathcal S^*(T-s)\,ds f,g\rangle,
\end{align*}
for $t\leq T$. 
Now, let $T=t$, and we find
\begin{equation}
Q_{X(t)}=\int_0^t\mathcal S(t-s)\E\left[\Gamma_Z(s)Q_B\Gamma_Z^*(s)\right]\mathcal S^*(t-s)\,ds.
\end{equation}
Note that
\begin{align*}
\Gamma_Z(s)Q_B\Gamma_Z^*(s)f&=\Gamma_Z(s)Q_B(\langle Y(s),f\rangle Z(s)) \\
&=\langle Y(s),f\rangle\Gamma_Z(s)(Q_B Z(s)) \\
&=\langle Y(s),f\rangle\langle Z(s),Q_B Z(s)\rangle Y(s) \\
&=|Q^{1/2}_BZ(s)|^2\langle Y(s),f\rangle Y(s) \\
&=|Q^{1/2}_BZ(s)|^2\mathcal V(s) f
\end{align*}
for any $f\in H$. Thus we recover the covariance functional that we can read off from Prop.~\ref{prop:ch-funct-OU};
\begin{equation}
\label{eq:covar-X(t)}
Q_{X(t)}=\int_0^t\mathcal S(t-s)\E\left[|Q^{1/2}_BZ(s)|^2\mathcal V(s)\right]\mathcal S^*(t-s)\,ds.
\end{equation} 
By Peszat and Zabczyk~\cite[Thm. 8.7 (iv)]{PZ} and the zero expectation of the stochastic integral with respect to $W$,
\begin{align*}
\E\left[\mathcal V(t) f,g\rangle\right]&=\E\left[\langle Y(t),f\rangle\langle Y(t),g\rangle\right] \\
&=\langle\mathcal U(t)Y_0,f\rangle\langle\mathcal U(t)Y_0,g\rangle+\E\left[\langle\int_0^t\mathcal U(t-s)\eta\,dW(s)^{\x2}f,g\rangle\right] \\
&=\langle(\mathcal U(t)Y_0)^{\x2}f,g\rangle+\langle\int_0^t\mathcal U(t-s)\eta Q_W\eta^*\mathcal U^*(t-s)\,ds f,g\rangle,
\end{align*}
for $f,g\in H$. Hence, in the particular case of $Z(t)=\gamma\in H$, with $|\gamma|=1$, we find that the covariance becomes
$$
Q_{X(t)}=\int_0^t\mathcal S(t-s)\left\{(\mathcal U(s)Y_0)^{\x2}+\int_0^s\mathcal U(u)\eta Q_W\eta^*\mathcal U^*(u)\,du\right\}\mathcal S^*(t-s)\,ds.
$$

We next apply our Ornstein-Uhlenbeck process $\{X(t)\}_{t\geq 0}$ with
tensor Heston stochastic volatility to the modelling of forward prices of commodity markets. 
For this purpose, we let $H$ be the Filipovic space $H_w$, which was introduced by Filipovic in \cite{filipovic}. For a measurable and increasing function $w:\mathbb{R_+\rightarrow\mathbb{R_+}}$ with $w(0)=1$ and $\int_0^{\infty} w^{-1}(x)dx<\infty$, the Filipovic space 
$H_w$ is defined as the space of absolutely continuous functions $f:\mathbb{R_+\rightarrow\mathbb{R}}$ such that
\[
|f|_w^2 := f(0)^2 + \int_0^\infty w(x)|f'(x)|^2dx <\infty.
\]
Here, $f'$ denotes the weak derivative of $f$. The space $H_w$ is a separable Hilbert space with inner product $\<\cdot,\cdot\>_w$ and associated norm $|\cdot|_w$.

We let $\{X(t)\}_{t\geq 0}$ be defined as in \eqref{X}, with $\mathcal C$ being the derivative operator $\partial/\partial x$. The derivative operator is densely defined on $H_w$ (see 
Filipovic~\cite{filipovic}), with the left-shift operator $\{\mathcal S(t)\}_{t\geq 0}$ being
its $C_0$-semigroup. For $f\in H_w$, the left-shift semigroup acts as $\mathcal{S}(t)f:= f(\cdot+t)\in H_w$.
Furthermore, we let $\delta_x:H_w\rightarrow\mathbb{R}$ be the evaluation functional, i.e. for $f\in H_w$ and $x\in\mathbb{R}_+$, $\delta_x(f):= f(x)$. We have that $\delta_x\in H_w^*$,
that is, the evaluation map is a linear functional on $H_w$. 
Letting $h_x\in H_w$ be given by
\[
 h_x(y) = 1 + \int_0^{x\wedge y} \frac{1}{w(x)}dx, \quad y\in\mathbb{R}_+,
\]
we have that $\delta_x = \<\cdot,h_x\>_w$ (see Benth and Kr\"uhner~\cite{BK-HJM}). 

The arbitrage-free price $F(t,T)$ at time $t$ of a contract delivering a commodity at a future time $T\geq t$, 
is modelled by $F(t,T):= \delta_{T-t}(X(t))=X(t,T-t)$ (see Benth and Kr\"uhner~\cite{BK-HJM}).
We adopt the Musiela notation and express the price in terms of {\it time to delivery}
$x\geq 0$ rather than {\it time of delivery} $T$, letting $f(t,x):= F(t,t+x) = \delta_x(X(t))=X(t,x)$. 
The next corollary gives the covariance between two contracts with different times to delivery. 
\begin{corollary}
For all $x,y\in\mathbb{R}_+$, we have
\begin{align*}
\textnormal{Cov}\big(f(t,x),\, f(t,y)\big)=\mathbb E\left[\int_0^t|Q_B^{1/2}Z(s)|_w^2Y(s,x+t-s)Y(s,y+t-s)\,ds\right],
\end{align*}
where $Y(s,z)=\delta_z(Y(s))$ for $z\in\mathbb R_+$.
\end{corollary}
\begin{proof}
Since $f(t,x)=\delta_x(X(t))=\langle X(t),h_x\rangle_w$, we find
$$
\textnormal{Cov}(f(t,x),f(t,y))=\textnormal{Cov}(\langle X(t),h_x\rangle_w,\langle X(t),h_y\rangle_w)=\langle Q_{X(t)}h_x,h_y\rangle_w,
$$
with $Q_{X(t)}$ given in \eqref{eq:covar-X(t)}. Since $\mathcal S^*(t)h_x=h_{x+t}$, it follows,
\begin{align*}
\langle\mathcal S(t-s)\mathcal V(s)\mathcal S^*(t-s) h_x,h_y\rangle_w&=\langle \mathcal S(t-s)\mathcal V(s)h_{x+t-s},h_y\rangle_w \\
&=\langle Y(s),h_{x+t-s}\rangle_w\langle\mathcal S(t-s)Y(s),h_y\rangle_w \\
&=\langle Y(s),h_{x+t-s}\rangle_w\langle Y(s),\mathcal S^*(t-s)h_y\rangle_w \\
&=Y(s,x+t-s)Y(s,y+t-s).
\end{align*}
The claim follows.
\end{proof}
The above corollary yields that the entire covariance structure between contracts with different times of maturity is determined by $Y$ only. We notice that the choice
of $\{Z(t)\}_{t\geq 0}$ in the definition of $\{\Gamma_Z(t)\}_{t\geq 0}$ only scales the covariance. Consider the special case of $Z(t)=\gamma\in H_w$. Using 
similar arguments to those in the derivation of $Q_{X(t)}$ yield,
\begin{align*}
&\mathbb E\left[Y(s,x+t-s)Y(s,y+t-s)\right] \\
&\qquad=\langle\mathcal U(s) Y_0,h_{x+t-s}\rangle_w\langle\mathcal U(s)Y_0,h_{y+t-s}\rangle_w \\
&\qquad\qquad+
\mathbb E\left[\langle\int_0^s\mathcal U(s-u)\eta\,dW(u),h_{x+t-s}\rangle_w\langle\int_0^s\mathcal U(s-u)\eta\,dW(u),h_{y+t-s}\rangle_w\right] \\
&\qquad=\langle\mathcal U(s) Y_0,h_{x+t-s}\rangle_w\langle\mathcal U(s)Y_0,h_{y+t-s}\rangle_w 
+\mathbb E\left[\langle\int_0^s\mathcal U(s-u)\eta\,dW(u)^{\otimes 2}h_{x+t-s},h_{y+t-s}\rangle_w\right] \\
&\qquad=\langle\mathcal U(s) Y_0,h_{x+t-s}\rangle_w\langle\mathcal U(s)Y_0,h_{y+t-s}\rangle_w
+\langle\int_0^s\mathcal U(u)\eta Q_W\eta^*\mathcal U^*(u)\,du h_{x+t-s},h_{y+t-s}\rangle_w.
\end{align*}
Thus, when $Z(t)=\gamma\in H_w$, we find the covariance to be
\begin{align*}
\textnormal{Cov}(f(t,x),f(t,y))&=|Q_B^{1/2}\gamma|_w^2\int_0^t\delta_{y+t-s}(\mathcal U(s)Y_0)^{\otimes 2}\delta_{x+t-s}^*(1)\,ds \\
&\qquad+|Q_B^{1/2}\gamma|_w^2\int_0^t\delta_{y+t-s}\int_0^s\mathcal U(u)\eta Q_W\eta^*\mathcal U^*(u)\,du\,\delta^*_{x+t-s}(1)\,ds,
\end{align*}
since $\delta_x^*(1)=h_x$. The covariance of $f(t,x)$ and $f(t,y)$ is determined by the parameters of the $Y$-process, more specifically, its volatility
$\eta$, the operator $\mathcal A$ (yielding a semigroup $\mathcal U$), the initial field $Y_0$ and the covariance operator $Q_W$ of the Wiener noise $W$ driving its dynamics. We also observe that the time integrals sample the parameters of $Y$ over the intervals $(x,x+t)$ and $(y,y+t)$ to build up the covariance of the
field $\{f(t,z)\}_{z\in\mathbb R_+}$, not only taking the values at $x$ and $y$ into account.

\section{The case when $\mathcal{A}$ is bounded} 

In this section, we analyse the tensor Heston stochastic variance process when 
$\mathcal{A}$ in \eqref{Y} is a bounded operator. Moreover, we make comparison with the 
classical Heston model on the real line (see Heston~\cite{H}) and discuss its 
extensions.

When $\mathcal A$ is bounded, \eqref{Y} has a strong solution and we can compute the dynamics of $\V(t)$ by an infinite dimensional version of It\^o's formula.
\begin{proposition} \label{prop:V_dynamics-Abounded}
Assume $\mathcal{A}$ is bounded. Then we have the following representation of $\V(t)$,
\begin{equation*} 
  \V(t) = \V(0) + \int_0^t \Phi(s) ds + \int_0^t \Psi(s) dW(s), t\geq 0,
\end{equation*}
where $\{\Phi(t)\}_{t\geq 0}$ is the $\mathcal H$-valued process
\begin{align*}
 \Phi(s) &= \mathcal{A}Y(s)\x Y(s) + Y(s)\x \mathcal{A}Y(s)+\eta Q_W\eta^*,
\end{align*}
and $\{\Psi(t)\}_{t\geq 0}$ is the $L(H,\mathcal H)$-valued process
\begin{align*}
 \Psi(s)(\cdot) &= \eta(\cdot)\x Y(s) + Y(s)\x\eta(\cdot).
\end{align*}
\end{proposition}
\begin{proof}
When $\mathcal{A}$ is bounded, the unique strong solution of \eqref{Y} is given by
\[
 Y(t) = Y_0 + \int_0^t \mathcal{A}Y(s)\,ds + \int_0^t \eta\, dW(s). 
\]
Define the function $v:H\rightarrow \mathcal{H}$ by $v(y):= y^{\x2}$ and observe that
$\V(t)=v(Y(t))$. To use the infinite dimensional It\^o formula by Curtain and Falb~\cite{CF}, 
we need to find the first and second order Frech\'et derivatives of $v$.
Define the functions $g_1:H\rightarrow L(H,\mathcal{H})$ and $g_2:H\rightarrow L(H,L(H,\mathcal{H}))$ by 
\begin{align*}
  g_1(y)& := \cdot\x y + y\x\cdot 
\end{align*}
and
\begin{align*}
  g_2(y)(h) &= h \x \cdot + \cdot \x h.
\end{align*}
A direct calculation reveals that,
\begin{align*}
 \| v(y+h) - v(y) - g_1(y)h \|_{\mathcal{H}} &= \| (y+h) \x (y+h) - y\x y - (h\x y + y\x h) \| \\
 &= \| y^{\x2} + y\x h + h\x y + h^{\x2} - y^{\x2} - h\x y - y\x h \| \\
 &= \| h^{\x2} \| \\
 &= |h|^2.
\end{align*}
Thus, we find,
\[
 \frac{\| v(y+h) - v(y) - g_1(y)h \|}{|h|} \leq \frac{|h|^2}{|h|} =|h|, 
\]
which converges to 0 when $|h|\rightarrow 0$. 
This shows that $g_1$ is the Frech\'et derivative of $v$, which we denote by $Dv$.
Next, for any $\xi\in H$, 
\begin{align*}
Dv(y+h)(\xi)&-Dv(y)(\xi)-g_2(y)(h)(\xi) \\
&=\xi\x (y+h)+(y+h)\x\xi-\xi\x y-y\x\xi-h\x\xi-\xi\x h=0,
\end{align*}
which shows that $g_2$ is the 
Frech\'et derivative of $Dv$, and hence the second order Frech\'et derivative of $v$, which we denote by $D^2 v$.
It follows from the infinite dimensional It\^o formula in Curtain and Falb~\cite{CF} that
\begin{align*}
 d\V(t) &= \left(Dv(Y(t))(\mathcal{A}Y(t)) + \frac{1}{2} \sum_{n=0}^{\infty} D^2 v(Y(t))(\eta(\sqrt{\lambda_n} e_n))(\eta(\sqrt{\lambda_n} e_n))\right) dt + Dv(Y(t))\eta\,dW(t), 
\end{align*}
where $\{e_n\}_{n\in\mathbb N}$ is an ONB of $H$ of eigenvectors of $Q_W$ 
with corresponding eigenvalues $\{\lambda_n\}_{n\in\mathbb N}$.
Inserting $g_1(Y(t))$ and $g_2(Y(t))$ for respectively $Dv(Y(t))$ and $D^2 v(Y(t))$, gives us
\begin{align*}
 d\V(t) =& \left(\mathcal{A}Y(t)\x Y(t) + Y(t)\x \mathcal{A}Y(t)\right)\,dt \\
&\qquad + \frac{1}{2} \sum_{n=0}^{\infty} \Big(\eta\big(\sqrt{\lambda_n} e_n\big)\x\eta\big(\sqrt{\lambda_n} e_n\big) + \eta\big(\sqrt{\lambda_n} e_n\big)\x\eta\big(\sqrt{\lambda_n} e_n\big)\Big)\,dt \\
 &\qquad+ \Big(\eta(\cdot)\x Y(t) + Y(t)\x\eta(\cdot)\Big) dW(t) \\
 =& \left(\mathcal{A}Y(t)\x Y(t) + Y(t)\x \mathcal{A}Y(t) + \sum_{n=0}^{\infty} \lambda_n\eta(e_n)^{\x2}\right)dt + \Psi(t) dW(t). 
\end{align*}
For $\xi\in H$, we find,
\begin{align*}
\eta Q_W\eta^*(\xi)&=\eta Q_W\sum_{n=1}^{\infty}\langle\eta^*(\xi),e_n\rangle e_n \\
&=\sum_{n=1}^{\infty}\langle \xi,\eta(e_n)\rangle\eta Q_W(e_n) \\
&=\sum_{n=1}^{\infty}\lambda_n\langle \xi,\eta(e_n)\rangle\eta(e_n) \\
&=\sum_{n=1}^{\infty}\lambda_n\eta(e_n)^{\x 2}(\xi).
\end{align*}
The proof is complete.
\end{proof}
We can formulate the dynamics of $\{\mathcal V(t)\}_{t\geq 0}$ as an operator-valued
stochastic differential equation.
\begin{proposition}
\label{prop:V-dynamics}
Assume that $\mathcal A$ is bounded. Then
\begin{align*}
d\mathcal V(t)&=\left(\mathcal A\mathcal V(t)\mathcal A^*+\mathcal V(t)-(\mathcal A-\text{Id})\mathcal V(t)(\mathcal A^*-\text{Id})+\eta Q_W\eta^*\right)\,dt \\
&\qquad+|\eta(\cdot)|\left(\Gamma_{\eta(\cdot)/|\eta(\cdot)|}(t)+\Gamma^*_{\eta(\cdot)/|\eta(\cdot)|}(t)\right)\,dW(t)
\end{align*}
where $\text{Id}$ is the identity operator on $H$ and $\Gamma_Z(t)$ is the Cholesky decomposition of $\mathcal V(t)$ defined in \eqref{def:gammaZ}. 
\end{proposition}  
\begin{proof}
For $y,f\in H$, we see from a direct computation that
\begin{align*}
((\mathcal A-\text{Id})y)^{\x 2}(f)&=\langle(\mathcal A-\text{Id})y,f\rangle(\mathcal A-\text{Id})y \\
&=\langle\mathcal A y,f\rangle\mathcal A y-\langle\mathcal A y,f\rangle y-\langle y,f\rangle
\mathcal A y +\langle y,f\rangle y \\
&=(\mathcal A y)^{\x 2}(f)-(\mathcal A y\x y)(f)-(y\x\mathcal A y)(f)+y^{\x 2}(f),
\end{align*}
or,
$$
\mathcal A y\x y+y\x\mathcal A y=(\mathcal A y)^{\x 2}+y^{\x 2}-((\mathcal A-\text{Id})y)^{\x 2}.
$$
Next, for any bounded operator $\mathcal L\in L(H)$, we have from linearity of
$\mathcal L$ that
\begin{align*}
(\mathcal L y)^{\x 2}(f)&=\langle\mathcal L  y,f\rangle\mathcal L y 
=\langle y,\mathcal L^*f\rangle\mathcal L y 
=\mathcal L(\langle y,\mathcal L^*f\rangle y) 
=\mathcal L(y^{\x 2}(\mathcal L^*f)) 
=\mathcal L y^{\x 2}\mathcal L^*(f).
\end{align*}
Thus,
$$
\mathcal A y\x y+y\x\mathcal A y=\mathcal A y^{\x 2}\mathcal A^*+y^{\x 2}
-(\mathcal A-\text{Id}) y^{\x 2} (\mathcal A^*-\text{Id}).
$$
With $y=Y(t)$ and recalling the definition of $\Phi(t)$ in 
Prop.~\ref{prop:V_dynamics-Abounded}, this shows the drift of $\mathcal V(t)$.

For $\xi,f\in H$, it holds that
$$
\Psi(t)(f)(\xi)=|\eta(f)|\left(\frac{\eta(f)}{|\eta(f)|}\x Y(t)+Y(t)\x\frac{\eta(f)}{|\eta(f)|}\right)(\xi)
$$ 
whenever $\eta(f)\neq 0$, with $\Psi(t)$ defined in Prop.~\ref{prop:V_dynamics-Abounded}.
The result follows.
\end{proof}
Recall from Lemma~\ref{lemma:gamma-prop} that $\mathcal V(t)=\Gamma_Z(t)\Gamma^*_Z(t)$. Hence, for any
$f\in H$, 
$$
\Gamma_{\eta(\cdot)/|\eta(\cdot)|}(t)\Gamma^*_{\eta(\cdot)/|\eta(\cdot)|}(t)(f)=
\Gamma_{\eta(f)/|\eta(f)|}(t)\Gamma^*_{\eta(f)/|\eta(f)|}(t)=\mathcal V(t)(f).
$$
Informally, we can say that the diffusion term of $\{\mathcal V(t)\}_{t\geq 0}$ is given as
the sum of the "square root" of $\{\mathcal V(t)\}_{t\geq 0}$ and its adjoint.

Let us consider our tensor Heston stochastic variance process in the particular case of finite dimensions, that is, $H=\R^d$ for $d\in\mathbb N$. 
We assume $\{W(t)\}_{t\geq 0}$ is a $d$-dimensional standard Brownian motion, 
and the $d$-dimensional stochastic process $\{Y(t)\}_{t\geq 0}$ is defined by the dynamics \eqref{Y} with 
$\mathcal A,\eta\in \R^{d\times d}$. It is straightforward to see that 
for any $x,y\in \R^d$, $x\x y=xy^{\top}$, where $y^{\top}$ means the transpose of $y$. Hence, $\mathcal V(t)=Y^{\x 2}(t)=Y(t)Y^{\top}(t)$. Moreover, if $x\in \R^d$, 
$$
\Psi(t)(x)=(\eta x)\x Y(t)+Y(t)\x(\eta x)=\eta xY^{\top}(t)+Y(t) x^{\top}\eta^{\top},
$$
and
$$
\mathcal A Y(t)\x Y(t)+Y(t)\x\mathcal A Y(t)=\mathcal A Y(t) Y^{\top}(t)+Y(t)(\mathcal A Y(t))^{\top}=\mathcal A\mathcal V(t)+\mathcal V(t)\mathcal A^{\top}.
$$
Hence, since $Q_W=I$, the $d\times d$ identity matrix, we find from Prop.~\ref{prop:V_dynamics-Abounded} that 
\begin{equation}
\label{finite-dim-tensor-Heston}
d\mathcal V(t)=\left(\eta\eta^{\top}+\mathcal A\mathcal V(t)+\mathcal V(t)\mathcal A^{\top}\right)\,dt+\eta\,dW(t)\, Y^{\top}(t)+Y(t)\,dW^{\top}(t)\,\eta^{\top}.
\end{equation}
This is a different dynamics than the Wishart processes on $\R^{d\times d}$ defined by Bru~\cite{Bru},
and proposed as a multifactor extension of the Heston stochastic volatility model in
Fonseca, Grasselli and Tebaldi~\cite{FGT}. The drift term in the Wishart process 
is analogous to the one in \eqref{finite-dim-tensor-Heston}, while the diffusion term in the Wishart process takes the form
$$
R\,d\bar{W}(t)\,\mathcal V^{1/2}(t)+\mathcal V^{1/2}(t)\,d\bar{W}^{\top}(t)\,R^{\top},
$$
where $\{\bar{W}(t)\}_{t\geq 0}$ is a $d\times d$ matrix-valued Brownian motion and $R$
is a $d\times d$ matrix. Our tensor model in infinite dimensions yields a simplified diffusion in
finite dimensions compared to the Wishart process of Bru~\cite{Bru}, where one is using
a Cholesky-type of representation of the square root of $\mathcal V(t)$, involving also
the "volatility" $\eta$ of the Ornstein-Uhlenbeck dynamics of $Y$.  To ensure 
a positive definite process, Bru~\cite{Bru} introduces strong conditions on
$\mathcal A$ and $R$, while our Heston model is positive definite by construction. 

Let us now slightly turn the perspective, going back to the general infinite dimensional
situation, and study the projection of the $\mathcal H$-valued process 
$\{\mathcal V(t)\}_{t\geq 0}$ to the real line in the sense of studying the process 
$\{\mathcal V(t)\}_{t\geq 0}$ expanded along a given element $f\in H$. 

To this end, for $f\in H$ introduce the linear functions $\mathcal L_f:\mathcal H\rightarrow\R$ 
by 
\begin{align*}
 \mathcal{L}_f(\mathcal{T}) := \<\<\mathcal{T},f^{\x2}\>\> &= \<\mathcal{T}(f),f\>.
\end{align*}
We note that for $h, g\in H$,
\begin{align} \label{proj_tensor}
 \mathcal{L}_f(h\x g) = \<\<h\x g, f^{\x2}\>\>= \<(h\x g)f, f\> = \<h,f\> \<g,f\>,
\end{align}
and, in particular, $\mathcal{L}_f(h^{\x2}) = \<h,f\>^2$. We define the real-valued
stochastic process $\{V(t;f)\}_{t\geq 0}$ as
\begin{equation}
\label{eq:dyn-projected-V}
V(t;f):=\mathcal L_f(\mathcal V(t))=\<Y(t),f\>^2,
\end{equation}
for $t\geq 0$. It is immediate from the definition that $\{V(t;f)\}_{t\geq 0}$ is
an $\mathcal F_t$-adapted process taking values on $\R_+$, the positive real line 
(including zero).
\begin{proposition}
\label{prop:V-dyn-realline}
Assume that $\mathcal{A}$ is bounded. Then the dynamics of $\{V(t;f)\}_{t\geq 0}$ defined
in \eqref{eq:dyn-projected-V} is 
\begin{align*}
dV(t;f)&=\left(V(t;f)+|Q^{1/2}_W\eta^*f|^2+\mathcal L_{\mathcal A^*f}(\V(t))-
\mathcal L_{(\mathcal A^*-\text{Id})f}(\V(t))\right)\,dt\\
 &\quad + 2|Q_W^{1/2} \eta^*f|\sqrt{V(t;f)}\, dw(t),
t\geq 0,
\end{align*}
where $w(t)$ is a real-valued Wiener process.
\end{proposition}
\begin{proof}
From Props.~\ref{prop:V_dynamics-Abounded} and \ref{prop:V-dynamics}, we have
\begin{align*}
dV(t;f)&=\left(\mathcal L_f(\mathcal A\V(t)\mathcal A^*)+V(t;f)-
\mathcal L_f((\mathcal A-\text{Id})\V(t)(\mathcal A^*-\text{Id}))+\mathcal L_f(\eta Q_W\eta^*)\right)\,dt \\
&\qquad+\mathcal L_f\left(\Psi(t)\,dW(t)\right).
\end{align*}
First,
$$
\mathcal L_f(\eta Q_W\eta^*)=\<\eta Q_W\eta^* f,f\>=|Q^{1/2}\eta^*f|^2.
$$
Next,
$$
\mathcal L_f(\mathcal A\V(t)\mathcal A^*)=\<\mathcal A\V(t)\mathcal A^*f,f\>
=\<\V(t)\mathcal A^*f,\mathcal A^*f\>=\mathcal L_{\mathcal A^*f}(\V(t)).
$$
This proves the drift term of $\{V(t;f)\}_{t\geq 0}$. 

We finally consider the projection of the stochastic integral.
From Thm.~2.1 in Benth and Kr\"uhner~\cite{BK-HJM}, 
\begin{align*}
 \mathcal{L}_{f}\left(\int_0^t \Psi(s)dW(s)\right) &= \int_0^t \sigma(s;f) dw(s),
\end{align*}
where $\{w(t)\}_{t\geq 0}$ is a real-valued Wiener process and $\sigma(t;f) =|Q^{1/2}_W\gamma(t;f)|$ with $\{\gamma(t;f)\}_{t\geq 0}$ being the 
$H$-valued stochastic process defined by 
$\mathcal{L}_{f}(\Psi(t)(\cdot)) = \<\gamma(t;f),\cdot\>$.
Since
\begin{align*}
 \mathcal{L}_{f}(\Psi(t)(\cdot)) &= \mathcal{L}_{f}(\eta(\cdot)\x Y(t)) + \mathcal{L}_{f}(Y(t)\x\eta(\cdot))\\
 &= \<\eta(\cdot),f\>\<Y(t),f\>+\<Y(t),f\>\<\eta(\cdot),f\> \\
 &= 2 \<\cdot,\eta^*f\>\<Y(t),f\>\\
 &= \<\cdot,2\<Y(t),f\>\eta^*f\>,
\end{align*}
we have $\gamma(t;f) = 2\<Y(t),f\>\eta^*f$. Observe in passing, recalling
Lemma~\ref{lem:finite-integral-norm},  that 
$\{\gamma(t;f)\}_{t\geq 0}$ is an $\mathcal F_t$-adapted stochastic process such that
$\E[\int_0^t\gamma^2(s;f)\,ds]<\infty$ for any $t>0$, and thus $w$-integrable. 
The integrand $\sigma(t;f)$ is therefore given by
\begin{align*}
 \sigma^2(t;f) &= |Q^{1/2}_W\gamma(t;f)|^2=4\<Y(t),f\>^2|Q^{1/2}_W\eta^*f|^2
=4\mathcal L_f(\V(t))|Q^{1/2}_W\eta^*f|^2.
\end{align*}
Thus, $\sigma(t;f)=2\sqrt{V(t;f)}|Q^{1/2}_W\eta^*f|$ and the proof is complete.
\end{proof}
We see that the process $\{V(t;f)\}_{t\geq 0}$ shares some similarities with a classical
real-valued Heston volatility model (see Heston~\cite{H}). $\{V(t;f)\}_{t\geq 0}$ has a square-root diffusion
term, and a linear drift term. However, there are also some additional drift terms
which are not expressible in $V(t;f)$. 

If $f\in H$ is an eigenvector of $\mathcal A^*$ with an eigenvalue $\lambda\in\R$, we find
that $\mathcal L_{\mathcal A^*f}(\V(t))=\lambda^2V(t;f)$ and 
$\mathcal L_{(\mathcal A^*-\text{Id})f}(\V(t))=(\lambda-1)^2V(t;f)$, and hence
by Prop.~\ref{prop:V-dyn-realline},
$$
dV(t;f)=\left(|Q^{1/2}_W\eta^*f|+2\lambda V(t;f)\right)\,dt+2|Q^{1/2}_W\eta^*f|\sqrt{V(t;f)}\,dw(t),
$$
which corresponds to a classical Heston stochastic variance process.

\end{document}